\documentclass{article} 

\usepackage{hyperref}

\usepackage[thmmarks]{ntheorem}
\usepackage{amsmath}
\usepackage{amsfonts}

\usepackage{amssymb}
\usepackage{tikz}       
\usetikzlibrary{calc,arrows}

{   
	\theoremstyle{nonumberplain}
	\theoremheaderfont{\bfseries}
	\theorembodyfont{\normalfont}
	\theoremsymbol{$\square$}
	\newtheorem{proof}{Proof}
}

\allowdisplaybreaks[3]

\newtheorem{definition}{Definition}[section]
\newtheorem{theorem}{Theorem}[section]
\newtheorem{corollary}[theorem]{Corollary}
\newtheorem{lemma}[theorem]{Lemma}
\newtheorem{proposition}[theorem]{Proposition}

\usepackage[bottom=25mm,left=25mm,right=25mm,top=25mm]{geometry}

\title{A class of non-matchable distributive lattices
}

\author{Xu Wang, Xuxu Zhao and Haiyuan Yao\footnote{Corresponding author.}%
	\\ {\footnotesize College of Mathematics and Statistics, Northwest Normal University, Lanzhou, Gansu 730070, PR China}}

\date{}

\begin{document}

\maketitle


\begin{abstract}
	The set of all perfect matchings of a plane (weakly) elementary bipartite graph equipped with a partial order is a poset, moreover the poset is a finite distributive lattice and its Hasse diagram is isomorphic to $Z$-transformation directed graph of the graph. A finite distributive lattice is matchable if its Hasse diagram is isomorphic to a $Z$-transformation directed graph of a plane weakly elementary bipartite graph, otherwise non-matchable. We introduce the meet-irreducible cell with respect to a perfect matching of a plane (weakly) elementary bipartite graph and give its equivalent characterizations. Using these, we extend a result on non-matchable distributive lattices, and obtain a class of new non-matchable distributive lattices.
	
	
	\textbf{2010 AMS Subj. Class.:} 05C70, 06D50 
\end{abstract}

\section{Introduction}

Zhang et al.\ \cite{aZhangGC1988b} introduced a concept of $Z$-transformation graph (called by some authors resonance graph) on the set of perfect matchings (or 1-factors) of hexagonal system; in addition, Zhang and Zhang \cite{aZhangZ2000} extended the concept to a general plane bipartite graph with a perfect matching and obtained some results on a plane (weakly) elementary bipartite graph. Let $G$ be a graph with a perfect matching, denote by $\mathcal{M}(G)$ the set of all perfect matchings of $G$. The $Z$-transformation directed graph $\vec{Z}(G)$ is an orientation of $Z$-transformation graph by orientating all the edges \cite{aZhangZ1999}. Lam and Zhang \cite{aLamZ2003} proved that $\mathcal{M}(G)$ equipped with a partial order is a finite distributive lattice and its Hasse diagram is isomorphic to $\vec{Z}(G)$. There are some results on finite distributive lattices and $Z$-transformation directed graphs \cite{aZhangZY2004c,aZhang2006,aZhangOY2009}. Recently, Zhang et al.\ \cite{aZhangYY2014} introduced the concept of matchable distributive lattice and got some consequences on matchable distributive lattices, Yao and Zhang \cite{aYaoZ2015} obtained some results on
non-matchable distributive lattices .

In the paper we first obtain Proposition~\ref{th:qkm} from the Proof of Lemma~3.7 in \cite{aZhangYY2007}. In a finite lattice, an element is \emph{meet-irreducible} if and only if it is covered by exactly one element. From a graphical point of view, if and only if there is exactly one arc (directed edge) to the vertex (element) in $\vec{Z}(G)$.
Consider the arc $f$ with its tail $M$, since $M$ and $f$ are perfect matching of $G$ and proper $M$-alternating cell, respectively, thus we call the cell meet-irreducible cell with respect to $M$. Furthermore, we have Theorem~\ref{th:mict} that is analogous to a lemma in \cite{aQiZ2010}. However, our method is completely different from their proof. Finally, by Theorem~\ref{th:mict}, we extend Theorem~4.8 in \cite{aYaoZ2015}, and obtain a class of non-matchable distributive lattices by Kuratowski's Theorem.

\section{Preliminaries}

A set $P$ equipped with a partial order relation $\le$ is said to be a \emph{partially ordered set} (poset for short). Given any poset $P$, the \emph{dual} $P^*$ of $P$ by defining $x\le y$ to hold in $P^*$ if and only if $y\le x$ holds in $P$. A poset $P$ is a \emph{chain} if any two elements of $P$ are comparable, and we write $\mathbf{n}$ to denote the chain obtained by giving $\{0,1,\dots,n-1\}$ the order in which $0<1<\cdots<n-1$. The set of all filters of a poset $P$ is denoted by $\mathcal{F}(P)$, and carries the usual anti-inclusion order; and the \emph{filter lattice} $\mathcal{F}(P)$ is a distributive lattice. A lattice is \emph{nontrivial} if it has at least two elements and a finite distributive lattice is \emph{irreducible} if it cannot be decomposed into a direct product of two nontrivial finite distributive lattices.

The symmetric difference of two finite sets $A$ and $B$ is defined as $A\oplus B := (A\cup B)\setminus(A\cap B)$. If $M$ is a perfect matching of a graph and $C$ is an $M$-alternating cycle of the graph, then the symmetric difference of $M$ and edge-set $E(C)$ is another perfect matching of the graph, which is simply denoted by $M\oplus C$. Let $G$ be a plane bipartite graph with a perfect matching, and the vertices of $G$ are colored properly black and white such that the two ends of every edge receive different colors.
An $M$-alternating cycle of $G$ is said to be \emph{proper}, if every edge of the cycle belonging to $M$ goes from white end-vertex to black end-vertex by the clockwise orientation of the cycle; otherwise \emph{improper} \cite{aZhangZ1997b}. An inner face of a graph is called a \emph{cell} if its boundary is a cycle, and we will say that the cycle is a cell too.

For some concepts and notations not explained in the paper, refer to \cite{bDaveyP2002,bStanl2011} for poset and lattice, \cite{bBondyM2008,bHarar1969} for graph theory.

Obverse that the $M$-alternating cycle intersecting a improper $M$-alternating cycle must be proper, \textit{vice versa}. Obviously, we have the following result.

\begin{lemma}[\cite{aZhangZY2004c}]\label{th:ipdj}
	If $G$ be a plane bipartite graph with a matching $M$, then any two proper (resp.\ improper) $M$-alternating cells are disjoint.
\end{lemma}

\begin{definition}[\cite{aZhangZ2000}]
	Let $G$ be a plane bipartite graph. The $Z$-transformation graph $Z(G)$ is defined on $\mathcal{M}(G)$: $M_1, M_2 \in \mathcal{M}(G)$ are joined by an edge if and only if $M_1\oplus M_2$ is a cell of $G$.
	And  $Z$-transformation digraph $\vec{Z}(G)$ is the orientation of $Z(G)$: an edge $M_1M_2$ of $Z(G)$ is oriented from $M_1$ to $M_2$ if $M_1\oplus M_2$ form a proper $M_1$-alternating (thus improper $M_2$-alternating) cell.
\end{definition}

An edge of graph $G$ is \emph{allowed} if it lies in a perfect matching of $G$. A graph $G$ is said to be \emph{elementary} if its allowed edges form a connected subgraph of $G$, then $G$ is connected and every edge of $G$ is allowed. A subgraph $H$ of $G$ is said to be \emph{nice} if $G - V(H)$ has a perfect matching \cite{bLovasP1986}. Let $G$ be a bipartite graph, from Theorem~4.1.1 in \cite{bLovasP1986}, we have that $G$ is elementary if and only if $G$ is connected and every edge of $G$ is allowed.

\begin{definition}[\cite{aZhangZ2000}]
	A bipartite graph $G$ is weakly elementary if the subgraph of $G$ consisting of $C$ together with its interior is elementary for every nice cycle $C$ of $G$.
\end{definition}

Let $G$ be a plane bipartite graph with a perfect matching, a binary relation $\le$ on $\mathcal{M}(G)$ is defined as: for $M_1,M_2\in\mathcal{M}(G)$, $M_1\le M_2$ if and only if $\vec{Z}(G)$ has a directed path from $M_2$ to $M_1$ \cite{aZhangZ2000}, thus $(\mathcal{M}(G);\le)$ is a poset \cite{aLamZ2003}. For convenient, we write $\mathcal{M}(G)$ for poset $(\mathcal{M}(G),\le)$.

\begin{theorem}[\cite{aLamZ2003}]\label{th:mfdl}
	If $G$ is a plane (weakly) elementary bipartite graph, then $\mathcal{M}(G)$ is a finite distributive lattice and its Hasse diagram is isomorphic to $\vec{Z}(G)$.
\end{theorem}

\begin{definition}[\cite{aZhangYY2014}]
	A finite distributive lattice $L$ is matchable if there is a plane weakly elementary bipartite graph $G$ such that $L\cong\mathcal{M}(G)$; otherwise it is non-matchable.
\end{definition}

\section{Meet-irreducible cell}

The Proof of Lemma~3.7 in \cite{aZhangYY2007} implies the following proposition.
\begin{proposition}\label{th:qkm}
	If $G$ is a plane elementary bipartite graph with a perfect matching $M$, then there exists a hypercube in $\vec{Z}(G)$ generated by some pairwise disjoint $M$-alternating cells. In particular, $M$ is the top (resp.\ bottom) of the hypercube in $\mathcal{M}(G)$ if these $M$-alternating cells are proper (resp.\ improper).
\end{proposition}

It is obvious that the dimension of the hypercube is equal to the number of these pairwise disjoint $M$-alternating cells. In particular, the hypercube is a quadrilateral if and only if it is generated by exactly two disjoint $M$-alternating cells in $G$ \cite{aZhangOY2009,aZhangZY2004c,aZhangYY2007}.

\begin{definition}
	Let $G$ be a plane (weakly) elementary bipartite graph with a perfect matching $M$. A meet-irreducible cell $f$ with respect to $M$ is a proper $M$-alternating cell if and only if $M\oplus f$ is meet-irreducible in $\mathcal{M}(G)$.
\end{definition}

\begin{theorem}\label{th:mict}
	Let $G$ be a plane (weakly) elementary bipartite graph $G$ with perfect matching $M$ and let $f$ be a proper $M$-alternating cell.
	\begin{enumerate}
		\itemsep=0em \parskip=0em   
		\item If $G$ has no improper $M$-alternating cell (namely, $M$ is the top of $\mathcal{M}(G)$), then every (proper) $M$-alternating cell is a meet-irreducible cell with respect to $M$;
		\item If $G$ has some improper $M$-alternating cells, then the following are equivalent:
		\begin{enumerate}
			\item the cell $f$ is a meet-irreducible cell with respect to $M$;
			\item the cell $f$ intersects every improper $M$-alternating cell;
			\item there is no perfect matching $M'$ in $V(\mathcal{Q})\setminus\{M\}$ such that $f$ is a proper $M'$-alternating cell, where $\mathcal{Q}$ is a hypercube generated by all improper $M$-alternating cells.
		\end{enumerate}
	\end{enumerate}
\end{theorem}

\begin{proof}
	1.~It is trivial by the definition of $Z$-transformation directed graph.
	
	2.~Firstly suppose that the cell $f$ is a meet-irreducible cell with respect to $M$, but there is at least one improper $M$-alternating cell $f'$ such that $f$ and $f'$ are disjoint. Thus $M\oplus f = ((M\oplus f')\oplus f)\oplus f'$, i.e.\ $G$ has two improper $M\oplus f$-alternating cells, hence $M\oplus f$ is not meet-irreducible, contradicting the supposition that $f$ is a meet-irreducible cell with respect to $M$.
	
	Next suppose that the cell $f$ intersects every improper $M$-alternating cell, but there is a perfect matching $M'$ in $V(\mathcal{Q})\setminus\{M\}$ such that $f$ is a proper $M'$-alternating cell. In fact, by Proposition~\ref{th:qkm}, there is at least one improper $M$-alternating cell $f'$ is a proper $M'$-alternating cell. 
	Hence $f$ and $f'$ are disjoint by Lemma~\ref{th:ipdj},
	a contradiction.
	
	Finally, suppose that there is no perfect matching $M'$ in $V(\mathcal{Q})\setminus\{M\}$ such that $f$ is a proper $M'$-alternating cell, but $f$ is not a meet-irreducible cell with respect to $M$. Thus $G$ has at least one improper $M\oplus f$-alternating cell $f'$ except $f$, by Lemma~\ref{th:ipdj}, hence $f$ and $f'$ are disjoint. Therefore $f'$ is an improper $M$-alternating cell, this means that $f$ is a proper $M\oplus f'$-alternating cell, i.e.\ there is a perfect matching $M'=M\oplus f'$ in $V(\mathcal{Q})\setminus\{M\}$ such that $f$ is a proper $M'$-alternating cell, a contradiction.
\end{proof}

If every proper $M$-alternating cell is a meet-irreducible cell with respect to $M$, then $M$ is a top of $\mathcal{M}(G)$ if $G$ has no improper $M$-alternating cell, otherwise cut vertex in $Z(G)$. Moreover we obtain the following corollary as a consequence of Theorem~\ref{th:mict}.

\begin{corollary}[\cite{aZhangZ1999,aZhangZY2004c}]\label{th:cutmic}
	If $G$ is a plane elementary bipartite graph with a perfect matching $M$, then $M$ is a cut vertex of $Z(G)$ if and only if $G$ has both proper and improper $M$-alternating cells and every proper $M$-alternating cell is a meet-irreducible cell with respect to $M$; i.e.\ every proper $M$-alternating cell intersects every improper $M$-alternating cell.
\end{corollary}

Note that duality of lattice, meet-irreducible cell, Theorem~\ref{th:mict} and Corollary~\ref{th:cutmic} could be treated in dual.

\section{Non-matchable distributive lattice}
Subdivide an edge $e$ is to delete $e$, add a new vertex $v$, and join $v$ to the ends of $e$.
Any graph derived from a graph $G$ by a sequence of edge subdivisions is called a subdivision of $G$.

\begin{theorem}[Kuratowski's Theorem]
	A graph is planar if and only if it contains no subdivision of either $K_5$ or $K_{3,3}$.
\end{theorem}

From the proof of Lemma~4.2 in \cite{aYaoZ2015} and Theorem~\ref{th:mict}, the following theorem is immediate.
\begin{theorem}\label{th:k33nm}
	Let $L$ be a finite distributive lattice and $x\in L$. If $x$ is covered by at least three elements and covers at least three meet-irreducible elements, then $L$ is non-matchable.
\end{theorem}

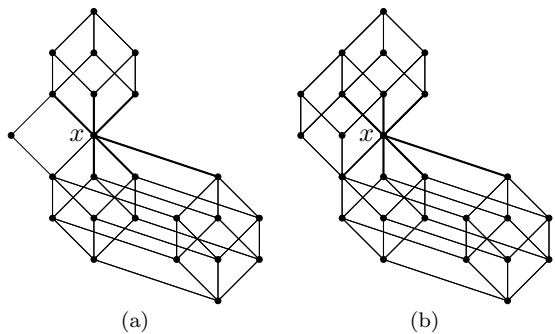
\begin{figure}[!ht]
	\centering
	\begin{tikzpicture}[scale=0.55]
	  \begin{scope}[xshift=-7cm]
	\foreach \i in {0,1}
	{
		\foreach \j in {0,1}
		{
			\foreach \k in {0,1}
			{
				\draw (\i-\j,\i+\j) -- (\i-\j,\i+\j+1)
				(\i,\i+\k) -- (\i-1,\i+1+\k)
				(-\j,\j+\k) -- (1-\j,1+\j+\k);
				\filldraw (\i-\j,\i+\j+\k) circle (2pt);
			}
		}
	}
	
	\foreach \i in {0,1}
	{
		\foreach \j in {0,1}
		{
			\foreach \k in {0,1}
			{
				\foreach \l in {0,1}
				{
					\draw (\j+3*\l,-\j-\k-\l) -- (-1+\j+3*\l,-1-\j-\k-\l)
					(-\i+3*\l,-\i-\k-\l) -- (-\i+1+3*\l,-\i-1-\k-\l)
					(-\i+\j+3*\l,-\i-\j-\l) -- (-\i+\j+3*\l,-\i-\j-1-\l)
					(-\i+\j,-\i-\j-\k) -- (-\i+\j+3,-\i-\j-\k-1);
					\filldraw (-\i+\j+3*\l,-\i-\j-\k-\l) circle (2pt);
				}
			}
		}
	}
	
	\draw (-1,1) -- (-2,0) -- (-1,-1);
	\filldraw (-2,0) circle (2pt);
	
	\draw[thick] (-1,1) -- (0,0) -- (3,-1)
	(0,1) -- (0,0) -- (1,-1)
	(1,1) -- (0,0) -- (0,-1);
	
	\node[left] at (0,0) {$x$};
	
	\node at (1,-4.5) {\footnotesize{(a)}};
	
	\end{scope}
	
	\foreach \i in {-1,0,1}
	{
		\foreach \j in {0,1}
		{
			\foreach \k in {0,1}
			{
				\draw (\i-\j,\i+\j) -- (\i-\j,\i+\j+1)
				(\i,\i+\k) -- (\i-1,\i+1+\k);
				\foreach \t in {0,1}
				{
					\draw (\t-\j,\t+\j+\k) -- (\t-1-\j,\t-1+\j+\k);
				}
				\filldraw (\i-\j,\i+\j+\k) circle (2pt);
			}
		}
	}
	
	\foreach \i in {0,1}
	{
		\foreach \j in {0,1}
		{
			\foreach \k in {0,1}
			{
				\foreach \l in {0,1}
				{
					\draw (\j+3*\l,-\j-\k-\l) -- (-1+\j+3*\l,-1-\j-\k-\l)
					(-\i+3*\l,-\i-\k-\l) -- (-\i+1+3*\l,-\i-1-\k-\l)
					(-\i+\j+3*\l,-\i-\j-\l) -- (-\i+\j+3*\l,-\i-\j-1-\l)
					(-\i+\j,-\i-\j-\k) -- (-\i+\j+3,-\i-\j-\k-1);
					\filldraw (-\i+\j+3*\l,-\i-\j-\k-\l) circle (2pt);
				}
			}
		}
	}
	
	\draw[thick] (-1,1) -- (0,0) -- (3,-1)
	(0,1) -- (0,0) -- (1,-1)
	(1,1) -- (0,0) -- (0,-1);
	
	\node[left] at (0,0) {$x$};
	
	\node at (1,-4.5) {\footnotesize{(b)}};
	\end{tikzpicture}\\
	\caption{Two non-matchable distributive lattices}\label{fig:p3s12}
\end{figure}

For instance, it is easy to see that each distributive lattice in Figure~\ref{fig:p3s12} is non-matchable by Theorem~\ref{th:k33nm}, but it is difficult to determine only by Theorem~4.3 in \cite{aYaoZ2015}.

Obviously, theorem~\ref{th:k33nm} could be obtained in dual.

\begin{corollary}
	If $L$ is a matchable distributive lattice, then for every element of $L$, it either is covered by at most two elements or covers at most two meet-irreducible elements in both $L$ and $L^*$.
\end{corollary}

Given a plane graph $G$, its (geometric) dual $G^*$ is constructed as follows: place a vertex in each face of $G$ (including the exterior face) and, if two faces have an edge $e$ in common, join the corresponding vertices by an edge $e^*$ crossing only $e$ \cite{bHarar1969}. It is easy to see that the dual $G^*$ of a plane graph $G$ is itself a planar graph\cite{bBondyM2008}.
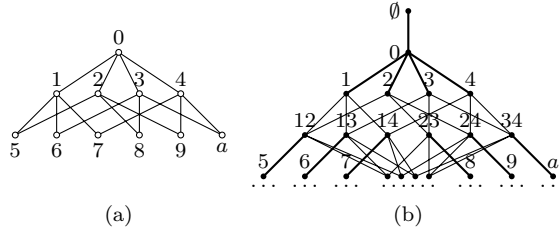
\begin{figure}[!ht]
	\centering\footnotesize{
		\begin{tikzpicture}[scale=0.55]
		\begin{scope}[xshift=-7cm]

		\coordinate[label=below:$5$] (3) at (0,0);
		\coordinate[label=below:$6$] (5) at (1,0);
		\coordinate[label=below:$7$] (6) at (2,0);
		\coordinate[label=below:$8$] (9) at (3,0);
		\coordinate[label=below:$9$] (A) at (4,0);
		\coordinate[label=below:$a$] (C) at (5,0);
		\coordinate[label=above:$1$] (7) at (1,1);
		\coordinate[label=above:$2$] (B) at (2,1);
		\coordinate[label=above:$3$] (D) at (3,1);
		\coordinate[label=above:$4$] (E) at (4,1);
		\coordinate[label=above:$0$] (F) at (2.5,2);
		
		\draw (7) -- (3)
		(7) -- (5)
		(7) -- (6)
		(7) -- (F)
		(B) -- (3)
		(B) -- (A)
		(B) -- (9)
		(B) -- (F)
		(D) -- (9)
		(D) -- (C)
		(D) -- (5)
		(D) -- (F)
		(E) -- (C)
		(E) -- (6)
		(E) -- (A)
		(E) -- (F);
		
		\foreach \i in {1,2,3,4}
		{
			\foreach \j in {0,1}
			{
				\filldraw[fill=white] (\i,\j) circle (2pt);
			}
		}
		\filldraw[fill=white] (F) circle (2pt)
		(3) circle (2pt)
		(C) circle (2pt);
		
		\node at (2.5,-1.95) {(a)};
		
		\end{scope}

		\coordinate[label=left:$\emptyset$] (es) at (2.5,3);
		\coordinate[label=left:$0$] (F) at (2.5,2);
		\coordinate[label=above:$1$] (7) at (1,1);
		\coordinate[label=above:$2$] (B) at (2,1);
		\coordinate[label=above:$3$] (D) at (3,1);
		\coordinate[label=above:$4$] (E) at (4,1);
		\coordinate[label=above:$12$] (7B) at (0,0);
		\coordinate[label=above:$13$] (7D) at (1,0);
		\coordinate[label=above:$14$] (7E) at (2,0);
		\coordinate[label=above:$23$] (BD) at (3,0);
		\coordinate[label=above:$24$] (BE) at (4,0);
		\coordinate[label=above:$34$] (DE) at (5,0);
		\coordinate[label=above:$5$] (3) at (-1,-1);
		\coordinate[label=above:$6$] (5) at (0,-1);
		\coordinate[label=above:$7$] (6) at (1,-1);
		\coordinate (7BD) at (2,-1);
		\coordinate (7BE) at (2.33,-1);
		\coordinate (7DE) at (2.67,-1);
		\coordinate (BDE) at (3,-1);
		\coordinate[label=above:$8$] (9) at (4,-1);
		\coordinate[label=above:$9$] (A) at (5,-1);
		\coordinate[label=above:$a$] (C) at (6,-1);
		
		\foreach \k in {-1,0,1,2.15,2.85,4,5,6}
		{
			\node[below] at (\k,-1) {$\ldots$};
		}
		
		\draw[thick] (es) -- (F)
		(B) -- (F) -- (D)
		(7) -- (F) -- (E)
		(7B) -- (3)
		(7D) -- (5)
		(7E) -- (6)
		(BD) -- (9)
		(BE) -- (A)
		(DE) -- (C);
		\draw (7) -- (7B) -- (B) -- (BD) -- (D) -- (DE) -- (E) -- (7E) -- (7) -- (7D) -- (D)
		(B) -- (BE) -- (E)
		(BDE) -- (DE) -- (7DE) -- (7E) -- (7BE) -- (7B) -- (7BD) -- (BD) -- (BDE) -- (BE) -- (7BE)
		(7DE) -- (7D) -- (7BD);
		
		\foreach \i in {1,2,3,4}
		{
			\foreach \j in {0,1}
			{
				\fill (\i,\j) circle (2pt);
			}
		}
		\foreach \i in {-1,...,2,2.33,2.67,3,4,...,6}
		{
			\fill (\i,-1) circle (2pt);
		}
		\fill (0,0) circle (2pt)
		(5,0) circle (2pt)
		(2.5,3) circle (2pt)
		(2.5,2) circle (2pt);
		
		\node[below] at (2.5,-1.5) {(b)};
		\end{tikzpicture}}\\
	\caption{(a) The poset $\Delta$ and (b) a part of $\mathcal{F}(\Delta)$}\label{fig:dpfd}
\end{figure}
\begin{theorem}\label{th:fdnm}
	The distributive lattice $\mathcal{F}(\Delta)$ is non-matchable, where $\Delta$ is a poset as shown in Figure~\ref{fig:dpfd}(a).
\end{theorem}

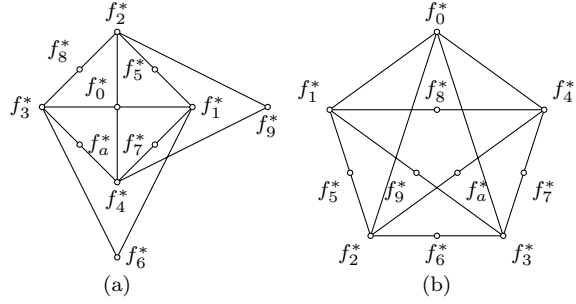
\begin{figure}[!ht]
	\centering\footnotesize{
		\begin{tikzpicture}[scale=0.5]
		  \begin{scope}[xshift=-8.5cm,yshift=1cm]
		\coordinate[label=above left:$f_0^*$] (F) at (0:0);
		\coordinate[label=right:$f_1^*$] (7) at (0:2);
		\coordinate[label=above:$f_2^*$] (B) at (90:2);
		\coordinate[label=left:$f_3^*$] (D) at (180:2);
		\coordinate[label=below:$f_4^*$] (E) at (270:2);
		\coordinate[label=left:$f_5^*$] (3) at (45:1.414);
		\coordinate[label=right:$f_6^*$] (5) at (270:4);
		\coordinate[label=left:$f_7^*$] (6) at (315:1.414);
		\coordinate[label=above left:$f_8^*$] (9) at (135:1.414);
		\coordinate[label=below:$f_9^*$] (A) at (0:4);
		\coordinate[label=right:$f_a^*$] (C) at (225:1.414);
		
		\draw  (F) -- (7) -- (B) -- (D) -- (E) -- (F) -- (B) -- (A) -- (E) -- (7) -- (5) -- (D) -- cycle;
		
		\foreach \i in {0,1,2,3}
		{
			\filldraw[fill=white] (90*\i:2) circle (2pt)
			(45+90*\i:1.414) circle (2pt);
		}
		\filldraw[fill=white] (0,0) circle (2pt)
		(4,0) circle (2pt)
		(0,-4) circle (2pt);
		
		\node at (0,-4.75) {(a)};
		
		\end{scope}
		\coordinate[label=90:$f_0^*$] (F) at (90:3);
		\coordinate[label=90+72:$f_1^*$] (7) at (90+72:3);
		\coordinate[label=90+72*2:$f_2^*$] (B) at (90+72*2:3);
		\coordinate[label=90+72*3:$f_3^*$] (D) at (90+72*3:3);
		\coordinate[label=90+72*4:$f_4^*$] (E) at (90+72*4:3);
		\coordinate[label=126+72:$f_5^*$] (3) at (126+72:2.427);
		\coordinate[label=126+72*2:$f_6^*$] (9) at (126+72*2:2.427);
		\coordinate[label=126+72*3:$f_7^*$] (C) at (126+72*3:2.427);
		\coordinate[label=90:$f_8^*$] (6) at (90:0.927);
		\coordinate[label=90+72*2:$f_9^*$] (5) at (90+72*2:0.927);
		\coordinate[label=90+72*3:$f_a^*$] (A) at (90+72*3:0.927);
		
		\draw (F) -- (7) -- (B) -- (D) -- (E) -- (F) -- (B) -- (E) -- (7) -- (D) -- cycle;
		
		\foreach \i in {0,...,4}
		{
			\filldraw[fill=white] (90+72*\i:3) circle (2pt);
		}
		\foreach \i in {1,2,3}
		{
			\filldraw[fill=white] (126+72*\i:2.427) circle (2pt);
		}
		\foreach \i in {0,2,3}
		{
			\filldraw[fill=white] (90+72*\i:0.927) circle (2pt);
		}
		
		\node at (0,-3.75) {(b)};
		\end{tikzpicture}}\\
	\caption{Proof of Theorem~\ref{th:fdnm}}\label{fig:iopnm}
\end{figure}

\begin{proof}
	Recall that $\mathcal{F}(\Delta)$ is a finite distributive lattice. Suppose that $\mathcal{F}(\Delta)$ is matchable, since $\mathcal{F}(\Delta)$ is irreducible, then there exists a plane elementary bipartite graph $G$ such that $\vec{Z}(G)\cong\mathcal{F}(\Delta)$ \cite{aZhangYY2014}.
	
	Consider a part of $\mathcal{F}(\Delta)$ as drawn in Figure~\ref{fig:dpfd}(b), the vertices $\emptyset$, $0$, $1$, $\cdots$, $a$ correspond to the perfect matchings $M_{\emptyset}$, $M_0$, $M_1$, $\cdots$, $M_a$ of $G$, respectively. Let $f_0=M_{\emptyset}\oplus M_0$, $f_1=M_0\oplus M_1$, $f_5=M_{12}\oplus M_5$, $f_6=M_{13}\oplus M_6$, \dots, and $f_a=M_{34}\oplus M_a$. By definition of $Z$-transformation graph, then $f_0$ is a nice cell, so are $f_1$, $\cdots$, $f_a$. Since the cells $f_0$, $f_1$, $\cdots$, $f_a$ are meet-irreducible cells, by Theorem~\ref{th:mict}(2), the cell $f_0$ intersects $f_1$, $f_2$, $f_3$ and $f_4$; the cell $f_5$ intersects $f_1$  and $f_2$, but it does not intersect $f_3$ or $f_4$, because $f_5$, $f_3$ and $f_4$ are proper $M_{12}$-alternating cells. Thus $f_0$ and $f_5$ are distinct; analogously, $f_0$ and $f_i\ (i\in\{6,7,8,9,a\})$ are distinct too.
	
	Next, consider the dual $G^*$ of $G$,  as drawn in Figure~\ref{fig:iopnm}(a), vertex $f_0^*$ is adjacent with $f_1^*, f_2^*, f_3^*\ \text{and}\ f_4^*$, and $f_5^*$ is adjacent with $f_1^*$ and $f_2^*$, etc. Therefore, let $V'=\{f_0^*,\ \cdots,\ f_a^*\}$, thus $G^*$ contains a subgraph $S^* := G^*[V']$.
	Clearly $S^*$ (see Figure~\ref{fig:iopnm}(b)) is a subdivision of $K_5$. By Kuratowski's Theorem, hence $S^*$ is non-planar, contradicting the planarity of $G$.
\end{proof}

As a straightforward consequence of Theorem~\ref{th:fdnm}, we have the following result.
\begin{corollary}
	If a poset $P$ contains $\Delta$ as a convex sub-poset, then distributive lattice $\mathcal{F}(P)$ is non-matchable.
\end{corollary}

Clearly, for any finite distributive lattice $L$, the Cartesian product, linear sum and vertical sum of $\mathcal{F}(P)$ and $L$ are non-matchable.
In particular, the following corollary is immediate.
\begin{corollary}
	The distributive lattice $\mathcal{F}(\mathbf{2}^4)$ is non-matchable.
	In addition, the distributive lattice $\mathcal{F}\left(\prod_{j=1}^k\mathbf{n}_j\right)$ is non-matchable, where $k\ge 4$, $\mathbf{n}_j$ is a chain of length $n_j$ and $n_j\ge 2$ for every $j=1,2,\cdots,k$.
\end{corollary}




\end{document}